\RequirePackage{tikz}
\documentclass[pdflatex,sn-mathphys]{sn-jnl}

\usepackage{longtable}
\usepackage{graphicx}
\usepackage{xcolor}
\usepackage{amsmath}
\usepackage{amssymb}
\usepackage{amsthm}
\usepackage{mathtools}
\usepackage{subcaption}
\usepackage{booktabs}
\usepackage{multirow}
\usepackage{float}
\usepackage{url}

\jyear{2022}%

\usetikzlibrary{shapes,patterns}
\tikzset{set/.style={shape=rectangle, rounded corners, align=center, fill=gray!25, draw=gray!50}}
\tikzset{element/.style={shape=circle, draw, align=center, inner sep=0pt, minimum size=.8cm}}
\tikzset{elementw/.style={shape=circle, draw, align=center, inner sep=0pt, minimum size=.8cm}}
\tikzset{elementg1/.style={shape=circle, draw, align=center, inner sep=0pt, minimum size=.8cm, fill=gray!15}}
\tikzset{elementgm/.style={shape=circle, draw, align=center, inner sep=0pt, minimum size=.8cm, fill=gray!40}}
\tikzset{elementg2/.style={shape=circle, draw, align=center, inner sep=0pt, minimum size=.8cm, fill=gray!90}}
\tikzset{elementb/.style={shape=circle, draw, align=center, inner sep=0pt, minimum size=.8cm, fill=black}}
\tikzset{fill_black/.style={fill=black}}
\tikzset{fill_gray1/.style={fill=gray!15}}
\tikzset{fill_graym/.style={fill=gray!40}}
\tikzset{fill_gray2/.style={fill=gray!90}}
\tikzset{root/.style={shape=rectangle, draw, align=center, inner sep=0pt, minimum size=.80cm}}
\tikzset{image/.style={shape=circle, draw=none, align=center, inner sep=0pt, minimum size=.40cm, fill=gray!70}}
\tikzset{rootimage/.style={shape=rectangle, draw, align=center, inner sep=0pt, minimum size=.40cm, fill=gray!70}}

\theoremstyle{thmstyleone}%
\newtheorem{theorem}{Theorem}[section]
\newtheorem{lemma}{Lemma}[section]

\theoremstyle{thmstyletwo}%

\theoremstyle{thmstylethree}%
\newtheorem{definition}{Definition}%

\begin{document}

\title[On the arboreal jump number of a poset]{On the arboreal jump number of a poset\footnote{\textit{Preprint submitted to Order}}}


\author[1]{\fnm{Evellyn} \sur{S. Cavalcante}}\email{evellynsoares@dcc.ufmg.br}

\author*[2]{\fnm{Sebastián} \sur{Urrutia}}\email{sebastian.a.urrutia@himolde.no}

\author[1]{\fnm{Vinicius} \sur{F. dos Santos}}\email{viniciussantos@dcc.ufmg.br}

\affil[1]{\orgdiv{Department of Computer Science}, \orgname{Federal University of Minas Gerais - UFMG}, \orgaddress{\street{Av. Presidente Antônio Carlos, 6627, Pampulha}, \city{Belo Horizonte}, \postcode{31270-901}, \state{Minas Gerais}, \country{Brazil}}}

\affil*[2]{\orgdiv{Faculty of Logistic}, \orgname{Molde University College}, \orgaddress{\street{Britvegen 2}, \city{Molde}, \postcode{NO-6410}, \state{M{\o}re og Romsdal}, \country{Norway}}}


\abstract{
A jump is a pair of consecutive elements in an extension of a poset which are incomparable in the original poset.
The arboreal jump number is an NP-hard problem that aims to find an arboreal extension of a given poset with minimum number of jumps.
The contribution of this paper is twofold: (i)~a characterization that reveals a relation between the number of jumps of an arboreal order extension and the size of a partition of its elements that satisfy some structural properties of the covering graph; (ii)~a compact integer programming model and a heuristic to solve the arboreal jump number problem along with computational results comparing both strategies.
The exact method provides an optimality certificate for 18 out of 41 instances with execution time limited to two hours.
Furthermore, our heuristic was able to find good feasible solutions for all instances in less than three minutes.
}



\maketitle

\newpage

\textbf{Declarations}

\textbf{Funding Statement:} This study was financed in part by:

(E.C.) The Coordenação de Aperfeiçoamento de Pessoal de Nível Superior – Brasil (CAPES) – Finance Code 001.

(S.U.) None.

(V.S.) FAPEMIG, grant number APQ-01707-21, and CNPq, grant number 311679/2018-8. 

\textbf{Data Availability Statement:} Benchmark instances used in this manuscript are available upon request.

\textbf{Conflict of Interests Statement:} The authors have no conflict of interests on the content of this manuscript.

\textbf{Author Contribution Statement:} All three authors conceived the research question, developed the work, wrote the manuscript, and participated in the revision process. Evellyn Cavalcante coded the algorithmic approaches and performed the computational experiments.

\newpage

\section{Introduction}
\label{sec:introduction}
Partially ordered sets (posets) provide a natural way to represent precedence constraints over entities and are useful to understand the combinatorial structure of problems through analytical and graphical tools. One fundamental problem is the \textit{jump number problem}, that aims to find an optimal linear extension. An optimal linear extension is a total order of a poset respecting its precedence relations while minimizing the number of consecutive pairs of elements that are incomparable in the original poset. The decision version of this problem is $\mathcal{NP}$-complete \cite{syslo1984}.
The jump number problem is widely studied in the literature and has applications in routing and scheduling problems.

There are different approaches to solve the jump number problem, such as polynomial-time algorithms for easy cases  \cite{syslo1984,chein1980,duffus1982,colbourn1985,steiner1985,steiner1987,syslo1987,syslo1988}, algorithms based in decomposition and dynamic programming \cite{steiner1985,bianco1997}, approximation algorithms \cite{syslo1995,Felsner1990,Mitas1991,krysztowiak2013,yuan2015}, heuristics and meta-heuristics \cite{bianco1997,ngom1998,gambardella2000,montemanni2007,krysztowiak2015}, algorithm based on tree searches \cite{libralesso2020}, parameterized algorithms \cite{elzahar1984,kratsch2013} and integer programming \cite{kubo1991,balas1995,ascheuer2000,ahmed2001,gouveia2006,sherali2006,gouveia2018,mak2007}.

This paper focus on a generalization of the jump number problem, called arboreal jump number problem, which aims to find an arboreal extension having a minimum number of jumps. 
The arboreal jump number problem is equivalent to the jump number problem when the input poset has a maximum element.

The first work about the arboreal jump number problem proved its $\mathcal{NP}$-completeness, even for the restricted class of interval orders \cite{figueiredo2013}. In addition, the authors proposed a polynomial-time algorithm to find minimal arboreal extensions, characterized arboreal orders for N-free posets and provided an upper bound for the arboreal jump number.
In \cite{cavalcante2019}, the authors presented an exponential-size mathematical integer model to solve the problem.
To the best of our knowledge, those are the only works that tackled the arboreal jump number problem.
Another work worth mentioning characterized arboreal extensions using the concept of linear extensions and proposed a polynomial-time algorithm to generate arboreal extensions for a particular poset \cite{queiroz2005}.

This work contributes both theoretical and algorithmic to the literature of the arboreal jump number problem.
Our main contributions are:
\begin{itemize}
    \item introduce a relation between the number of jumps of an arboreal order extension and the size of a partition of the elements of the poset satisfying some structural properties. This relation is presented as a characterization of the instances of the problem that admit an arboreal extension with a certain number of jumps (section~\ref{sec:characterization});
    \item define and implement a compact integer programming model for the arboreal jump number problem with a polynomial number of constraints (section~\ref{sec:strategies});
    \item describe and implement a fast greedy heuristic based on an algorithm from the literature (section~\ref{sec:strategies});
    \item report a computational comparison of the performance of the proposed approaches, using the set of instances from TSPLIB \cite{tsplib} (section~\ref{sec:results}).
\end{itemize}

\subsection{Preliminaries}
A \textit{partial order} (or \textit{partially ordered set} or \textit{poset}) $\mathcal{P}$ is a pair $(V, R)$, where $V$ is a finite set, called \textit{ground set}, and $R$ is a reflexive, anti-symmetric and transitive binary relation on $V$.
It can be represented by a directed graph $G_{\mathcal{P}} = (V, A)$, in which each binary relation $(x,y) \in R$ corresponds to an arc from $x$ to $y$, that is, $(x,y) \in A$. 
Figure \ref{fig:dag} shows an example of a poset with its directed graph representation.

Two elements $x,y \in V$ are \textit{comparable}, $x \perp y$, if $(x,y) \in R$ or $(y,x) \in R$ and \textit{incomparable}, $x\|y$, otherwise.
Let $x,y \in V$, we say that $y$ \textit{covers} $x$ in $\mathcal{P}$, $x \prec_\mathcal{P} y$, when $(x,y) \in R$ and there is not a third element $z \in V$ such that $(x,z), (z,y) \in R$, for $x \neq y, x \neq z, y \neq z$. 
This is called \textit{covering relation}. 
We omit the subscript in case there is no doubt of which poset we refer.
A \textit{minimal element} of a poset is one that does not cover any other element.
If a poset has only a single minimal element, we called it a \textit{root} or \textit{minimum}. In such a case, we say the poset is \textit{rooted}.

The \textit{covering graph} $C_\mathcal{P} = (V, \hat{R})$ is associated with the covering relation and the arc $(x,y) \in \hat{R}$ if and only if $x \prec_\mathcal{P} y$.
Note that $C_\mathcal{P}$ is equivalent to the transitive reduction of $G_{\mathcal{P}}$, ignoring reflexivity. A \textit{Hasse diagram} is a representation of the covering graph of a poset in which the elements $x \in V$ are represented by points $p(x)$ of the plane satisfying two rules: (i)~if $(x,y) \in R$, $p(x)$ is below the horizontal line going through $p(y)$, (ii)~$p(x)$ and $p(y)$ are linked by a line segment if and only if $x \prec y$.
Figure \ref{fig:hasse_dag} is a Hasse diagram of the poset present in Figure \ref{fig:dag}.

\begin{figure}[htb]
	\centering
	\begin{subfigure}[t]{0.4\textwidth}
		\centering
        \scalebox{.7}{\begin{tikzpicture}
		\node [element] (1) at (-6, 1) {1};
		\node [element] (2) at (-4, 2) {2};
		\node [element] (3) at (-2, 3) {3};
		\node [element] (5) at (-4, 0) {4};
		\node [element] (6) at (-2, 1) {5};
		\node [element] (7) at (-6, -1) {6};
		\node [element] (8) at (0, 1) {7};
		\draw[->,-stealth] (1) to (2);
		\draw[->,-stealth] [bend left=45, looseness=0.75] (1) to (3);
		\draw[->,-stealth] (2) to (3);
		\draw[->,-stealth] (6) to (8);
		\draw[->,-stealth] (2) to (6);
		\draw[->,-stealth] [bend left=15, looseness=1.25] (1) to (8);
		\draw[->,-stealth] (1) to (6);
		\draw[->,-stealth] (1) to (5);
		\draw[->,-stealth] (5) to (6);
		\draw[->,-stealth] [bend left=15, looseness=1.25] (7) to (6);
		\draw[->,-stealth] (7) to (5);
		\draw[->,-stealth] [bend right=15, looseness=0.45] (7) to (8);
		\draw[->,-stealth] [bend left=15, looseness=1.25] (2) to (8);		
		\draw[->,-stealth] (5) to (8);
		\draw[->,-stealth] [in=135, out=45, loop] (1) to ();
		\draw[->,-stealth] [in=135, out=45, loop] (2) to ();
		\draw[->,-stealth] [in=135, out=45, loop] (3) to ();
		\draw[->,-stealth] [in=135, out=45, loop] (5) to ();
		\draw[->,-stealth] [in=135, out=45, loop] (6) to ();
		\draw[->,-stealth] [in=135, out=45, loop] (7) to ();
		\draw[->,-stealth] [in=135, out=45, loop] (8) to ();
\end{tikzpicture}}
        \caption{Graph representation of poset~$\mathcal{P}$}
        \label{fig:dag}
	\end{subfigure}
	\hfill
	\begin{subfigure}[t]{0.4\textwidth}
		\centering
        \scalebox{.7}{\begin{tikzpicture}
		\node [element] (1) at (-3, -2) {1};
		\node [element] (2) at (-3, 0) {2};
		\node [element] (3) at (-3, 2) {3};
		\node [element] (5) at (0, 0) {4};
		\node [element] (6) at (0, 2) {5};
		\node [element] (7) at (0, -2) {6};
		\node [element] (8) at (0, 3.5) {7};
		\draw (1) to (2);
		\draw (2) to (3);
		\draw (6) to (8);
		\draw (2) to (6);
		\draw (1) to (5);
		\draw (5) to (6);
		\draw (7) to (5);
\end{tikzpicture}}
        \caption{Hasse diagram of poset $\mathcal{P}$}
        \label{fig:hasse_dag}
	\end{subfigure}
	\caption{Graphical representations of poset $\mathcal{P} = (V,R)$, where $V=\{1,2,3,4,5,6,7\}$, $R~=~\{(1,1),(1,2),(1,3),(1,4),(1,5),(1,7),(2,2),(2,3),(2,5),(2,7),(3,3),(4,4),(4,5),$ $(4,7),(5,5),(5,7),$ $(6,6),(6,4),(6,5),(6,7)\}$}
\end{figure}
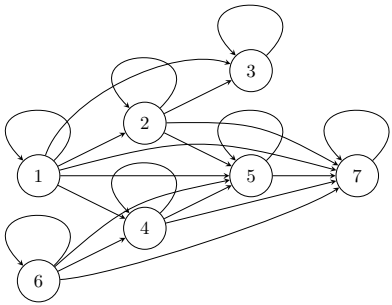
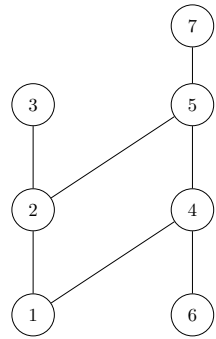

An \textit{induced subposet} of a given poset $\mathcal{P} = (V,R)$ is a poset $\mathcal{B} = (U,X)$, where $U \subseteq V$ and for all $x,y \in U$, $(x,y) \in X$ if and only if $(x,y) \in R$.
Figure~\ref{fig:induced_dag} shows an example of the subposet induced by elements $3, 4, 6$ and $7$ of the poset presented in Figure \ref{fig:dag}.

\begin{figure}[hbt]{}
    \centering
    \scalebox{.7}{\begin{tikzpicture}
		\node [element] (3) at (-2, 3) {3};
		\node [element] (4) at (-4, 0) {4};
		\node [element] (6) at (-6, -1) {6};
		\node [element] (7) at (0, 1) {7};
		\draw[->,-stealth] (6) to (4);
		\draw[->,-stealth] [bend right=15, looseness=0.45] (6) to (7);
		\draw[->,-stealth] (4) to (7);
		
		\draw[->,-stealth] [in=135, out=45, loop] (3) to ();
		\draw[->,-stealth] [in=135, out=45, loop] (4) to ();
		\draw[->,-stealth] [in=135, out=45, loop] (6) to ();
		\draw[->,-stealth] [in=135, out=45, loop] (7) to ();
\end{tikzpicture}}
    \caption[Representation of an induced subposet as a directed graph.]{Directed graph representing a subposet of poset $\mathcal{P} = (V,R)$ induced by elements $3, 4, 6$ and $7$.}
    \label{fig:induced_dag}
\end{figure}
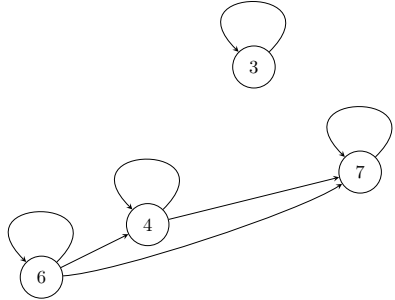

A poset $\mathcal{P} = (V,R)$ is a \textit{total order} (or \textit{linear order} or \textit{linearly ordered set} or \textit{chain}) if each pair $x,y$ of distinct elements of $V$ is comparable, i.e., either $(x,y) \in R$ or $(y,x) \in R$.
An \textit{arboreal order} is a poset with a minimum element satisfying that for all $x$ in $V$ the subposet induced by $\{y \mid (y,x)  \in R \}$ is a total order.

A poset is an arboreal order if and only if it does not contain any \textit{violation}.
There is a violation in $\mathcal{P}$ when $x \prec z$, $y \prec z$ for some $x,y,z \in V$, $x \neq y \neq z$. Element $z$ is called a \textit{violator}.
In Figure \ref{fig:hasse_dag} elements $1,4,6$ form a violation with element $4$ being the violator.

An order $\mathcal{Q} = (V, R')$ is an \textit{extension} of $\mathcal{P}$, if $R \subseteq R'$. In particular, if $R'$ does not have incomparable elements, then $\mathcal{Q}$ is a \textit{linear extension} of $\mathcal{P}$.
An \textit{arboreal extension} $\mathcal{A} = (V, R')$ of a partial order $\mathcal{P} = (V, R)$ is an extension that is an arboreal order. Note that in posets that have only one maximal (a maximum) element, an extension is arboreal if and only if it is linear.

A \textit{jump} in a extension $\mathcal{A} = (V,R^\prime)$ of $\mathcal{P}$ is a pair of covering elements $x \prec_\mathcal{A} y$ that is incomparable in $R$, $x\|_\mathcal{P}y$. 
The \textit{jump number} $s(\mathcal{P})$ of an order $\mathcal{P}$ is given by the linear extension having the minimum number of jumps.

The Jump Number Problem aims to find a linear extension $\mathcal{Q}$ of a poset $\mathcal{P}$ with $s(\mathcal{P})$ jumps.

The arboreal extension with the minimum number of jumps, that is the \textit{minimum arboreal extension}, gives the \textit{arboreal jump number} $s_a(\mathcal{P})$ of $\mathcal{P}$. 
The Arboreal Jump Number Problem aims to find an arboreal extension $\mathcal{A}$ of a poset $\mathcal{P}$ that contains $s_a(\mathcal{P})$ jumps.

Figures \ref{fig:linear_extension} and \ref{fig:arboreal_extension} show a linear extension and an arboreal extension, respectively, of the poset presented in Figure \ref{fig:hasse_dag}.

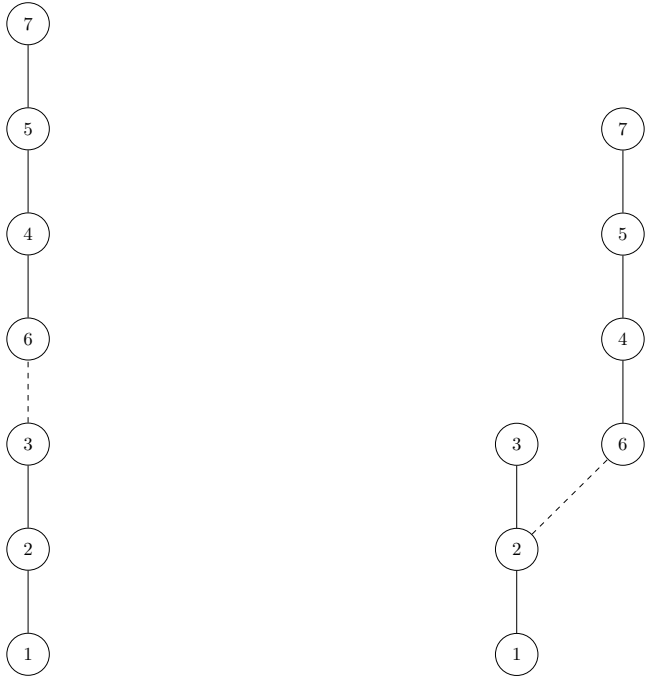
\begin{figure}
    \centering
    \begin{subfigure}[t]{.4\textwidth}
    	\centering
        \scalebox{.7}{\begin{tikzpicture}
		\node [element] (1) at (0, 0) {1};
		\node [element] (2) at (0, 2) {2};
		\node [element] (3) at (0, 4) {3};
		
		\node [element] (6) at (0, 6) {6};
		\node [element] (4) at (0, 8) {4};
		\node [element] (5) at (0, 10) {5};
		\node [element] (7) at (0, 12) {7};
		\draw[-] (1) to (2);
		\draw[-] (2) to (3);
		\draw[-, dashed] (3) to (6);
    	\draw[-] (6) to (4);
    	\draw[-] (4) to (5);
		\draw[-] (5) to (7);
\end{tikzpicture}}
        \caption{Hasse diagram of an optimal linear extension of $\mathcal{P}$. The jump number is one. This is also an optimal arboreal extension.}
        \label{fig:linear_extension}
    \end{subfigure}
    \hfill
    \begin{subfigure}[t]{.4\textwidth}
    	\centering
        \scalebox{.7}{\begin{tikzpicture}
		\node [element] (1) at (0, 1) {1};
		\node [element] (2) at (0, 3) {2};
		\node [element] (3) at (0, 5) {3};
		\node [element] (4) at (2, 7) {4};
		\node [element] (5) at (2, 9) {5};
		\node [element] (6) at (2, 5) {6};
		\node [element] (7) at (2, 11) {7};
		\draw[-] (1) to (2);
		\draw[-] (2) to (3);
		\draw[-, dashed] (2) to (6);
    	\draw[-] (6) to (4);
    	\draw[-] (4) to (5);
		\draw[-] (5) to (7);
\end{tikzpicture}}
        \caption{Hasse diagram of an optimal arboreal extension of $\mathcal{P}$. The arboreal jump number is one.}
        \label{fig:arboreal_extension}
    \end{subfigure}
    \caption{Hasse diagrams of optimal linear and arboreal extensions of poset $\mathcal{P} = (V,R)$}
\end{figure}

\section{Characterization}
\label{sec:characterization}
In this section, we present a relation between the number of jumps of an arboreal order extension and the size of a partition of the elements of the poset satisfying some structural properties. This relation is presented as a characterization of the instances of the problem that admit an arboreal extension with a certain number of jumps.

Our characterization unveils that an arboreal extension can be seen as a partition of the original poset that follows some well-define structural properties.
Even though this characterization only works for rooted posets, it will not be a limitation, since it is possible to add a root in any poset and easily find the corresponding arboreal extension of the original poset a posteriori. Lemma \ref{lm:uniqueRoot} below demonstrate this fact.

\begin{lemma}\label{lm:uniqueRoot}
Let $\mathcal{P}_i=(V_i, R_i), i = 1, \ldots, k$, be posets such that for all $(i,j), i \neq j, V_i \cap V_j = \emptyset$. 
Construct a poset $\mathcal{P} = (V,R)$, where $V = \bigcup_i V_i$ and $R = \bigcup_i R_i$.
Consider $\mathcal{Q} = (U,X)$ the poset in which $U = V \cup \{r\}$ and $X = R\cup \{(r,v) \mid v \in V\}$. 
Then $s_a(\mathcal{P}) = \sum_{i=1}^k s_a(\mathcal{P}_i) + k-1 = s_a(\mathcal{Q}) + k-1$.
\end{lemma}

\begin{proof}
Let $\mathcal{P}$ and $\mathcal{Q}$ be the posets defined according to the  statement.
Note that for the second equality it is enough to prove $\sum_{i=1}^k s_a(\mathcal{P}_i)=s_a(\mathcal{Q})$. 
Let $\mathcal{A} = (U,X^\prime)$ be an optimal arboreal extension for $\mathcal{Q}$. Note that $\mathcal{A}$ has $r$ as its root.
Define mixed jumps as the pairs $(x,y) \in X^\prime$ such that $x \in V_i$, $y \in V_j$, $i \neq j$. 
In the following, we prove that there is an optimal arboreal extension $\mathcal{A}$ without mixed jumps.

Suppose $\mathcal{A}$ is an optimal arboreal extension that minimizes the number of mixed jumps and includes at least one of those.
Let $(x,y)$ be a mixed jump with $x \in V_i$ and $y \in V_j$, such that for all $z$ satisfying $(y, z) \in X'$ it holds that $z \in V_j$.
Such a pair can be found iteratively as follows. Take any jump $(x,y)$ and suppose there is a $z \not\in V_j$ with $(y, z) \in X'$. But then, since $y$ and $z$ came from different posets, there must be another mixed jump $(x',y')$ in some chain from $y$ to $z$. Then we can take this jump as $(x,y)$ and repeat this process.
This process is finite, since $\mathcal{A}$ is a poset, and the desired jump is found.

With $(x,y)$ being the mixed jump with $x \in V_i$ and $y \in V_j$, let $v$ be an element such that $v \in V_j$, $(v,x) \in X^\prime$ and there is no $u \in V_j$ in which $\{(v,u),(u,x)\} \subset X^\prime$ or the root of $\mathcal{A}$ if such an element does not exist.
Remove $(x,y)$ from $X^\prime$ and add $(v,y)$ to obtain a different arboreal extension $\mathcal{B}$ of $\mathcal{Q}$. Note that $\mathcal{B}$ is in fact an arboreal extension of $\mathcal{Q}$ because the way in which the mixed jump $(x, y)$ and the element $v$ were selected. The pair $(v,y)$ is not a mixed jump because $v$ and $y$ belong to $V_j$. 
Then, the extension $\mathcal{B}$ contains either the same number of jumps of $\mathcal{A}$, if $(v,y) \notin R_j$, or one less, otherwise.
This contradicts the assumption that $\mathcal{A}$ is optimum and minimizes the number of mixed jumps. 
Therefore, there must exist an optimal extension of $\mathcal{Q}$ without mixed jumps, then $s_a(\mathcal{Q})=\sum_i^k s_a(\mathcal{P}_i)$. 

A similar argument shows that $s_a(\mathcal{P}) = \sum_i^k s_a(\mathcal{P}_i) + k-1$, observing that $k-1$ is the necessary and sufficient amount of mixed jumps to connect optimal arboreal extensions of each $P_i$ to obtain an optimal arboreal extension of $\mathcal{P}$. 

\end{proof}


Note that Lemma~\ref{lm:uniqueRoot} holds even if $k=1$. As a consequence of that result, we are assuming that every poset is rooted unless stated otherwise. Recall that the root of a poset will also be the root of any of its arboreal extensions. 

Consider the Definition \ref{def:contractingRelation} below that will be handful to the presentation of the characterization.

\begin{definition}\label{def:contractingRelation}
Let $\mathcal{P} = (V,R)$ be a poset and let $\pi = \{\pi_1, \ldots, \pi_s\}$ be a partition of $V$. The pair $\mathcal{P}_\pi=(\pi, R_\mathbf{\pi})$ is the relation induced by $\pi$, in which $(\pi_i, \pi_j) \in R_\mathbf{\pi}$ if and only if there are $x \in \pi_i$ and $y \in \pi_j$ such that $(x,y) \in R$.
\end{definition}

Now we state a characterization for a partial order and their arboreal extensions with $s$ jumps, considering a partition of its elements satisfying four properties.

\begin{theorem}\label{thm:characterization}
	A partial order $\mathcal{P} = (V, R)$ with root $r$ has  
	an arboreal extension $\mathcal{A} = (V,R')$ with $s$ jumps, $s \geq 1$, if and only if there is a partition $\pi = \{ \pi_1$, $\pi_2$, $\ldots$, $\pi_{s+1}\}$ of $V$ satisfying the following properties:

	\begin{enumerate}
		\item each $\pi_i$ induces an arboreal subposet $\mathcal{P}_i = (\pi_i, R_i)$ of $\mathcal{P}$ with root $r_i$;\label{it:induced_subposet}
		
		\item let $\mathcal{Q} = (\pi,T)$ be the relation induced by $\pi$ as defined in Definition \ref{def:contractingRelation} and $B$ be the transitive closure of $T$, then $\mathcal{P}_\pi = (\pi,B)$ is a partial order with root $r_{\pi}$;\label{it:rooted_poset}
		
		\item there is a function $f: \pi\setminus \{r_{\pi}\} \to V$ such that $f(\pi_j) \in \pi_i$, $i \neq j$ and $(f(\pi_j), r_j) \notin R$, corresponding to the jumps of $\mathcal{A}$. 
		Moreover $\mathcal{A}_\pi = (\pi, B^\prime)$ is an arboreal order extension of $\mathcal{P}_\pi$, where $B^\prime$ is the transitive closure of $\{(\pi_i, \pi_j)\mid f(\pi_j) \in \pi_i\} \cup \{(\pi_i, \pi_i)\mid \pi_i \in \pi\}$;\label{it:f_function}
		
		\item let $g: \pi \to 2^V$ be a mapping satisfying\label{it:recursion} 
	
\begin{equation*}
g(\pi_i)=\begin{cases}
         \emptyset & \text{if } \pi_i = r_{\pi} \\
         \{f(\pi_i)\} \cup g(\pi_j), f(\pi_i)\in \pi_j & \text{otherwise,} \\
     \end{cases}
\end{equation*}
then for all $x \prec_\mathcal{P} y$ such that $x \in \pi_u$, $y \in \pi_v, u \neq v$ there is $z \in g(\pi_v)$ such that $(x,z) \in R_u$.
		
	\end{enumerate}
\end{theorem}

To guide the reader through the proof, we are going to use an example of a poset $\mathcal{P}$ on 16 elements showed in Figure \ref{fig:characterization}. 

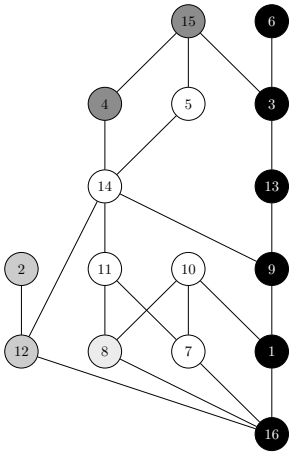
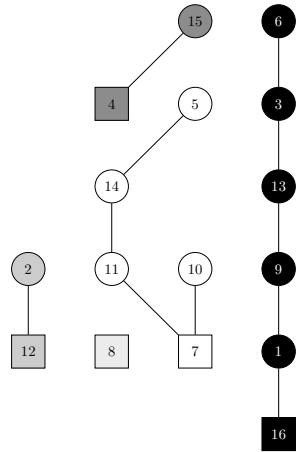
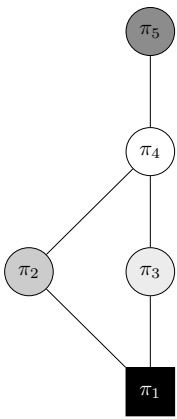
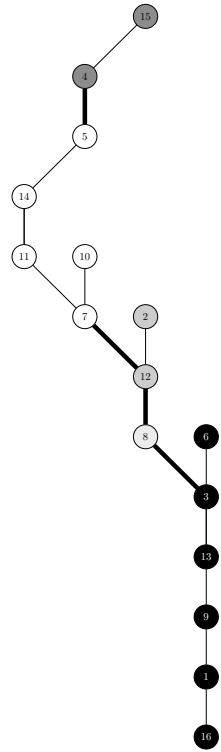
\begin{figure}[H]
    \centering
    \begin{subfigure}[t]{0.4\textwidth}
        \centering
        \scalebox{.55}{\begin{tikzpicture}

\node[elementb] (6) at (6,10) {\color{white}6};
\node[elementb] (3) at (6,8) {\color{white}3};
\node[elementb] (13) at (6,6)  {\color{white}13};
\node[elementb] (9) at (6,4) {\color{white}9};
\node[elementb] (1) at (6,2) {\color{white}1};
\node[elementb] (16) at (6,0)  {\color{white}16};

\node[elementg2] (15) at (4,10) {15};
\node[elementw] (5) at (4,8) {5};
\node[elementw] (10) at (4,4) {10};
\node[elementw] (7) at (4,2)  {7};

\node[elementg2] (4) at (2,8)  {4};
\node[elementw] (14) at (2,6) {14};
\node[elementw] (11) at (2,4) {11};
\node[elementg1] (8) at (2,2)  {8};

\node[elementgm] (2) at (0,4) {2};
\node[elementgm] (12) at (0,2) {12};

\draw[-,>=stealth] (16) -- (1);
\draw[-,>=stealth] (16) -- (7);
\draw[-,>=stealth] (16) -- (8);
\draw[-,>=stealth] (16) -- (12);

\draw[-,>=stealth] (1) -- (10);
\draw[-,>=stealth] (1) -- (9);

\draw[-,>=stealth] (7) -- (10);
\draw[-,>=stealth] (7) -- (11);

\draw[-,>=stealth] (8) -- (10);
\draw[-,>=stealth] (8) -- (11);

\draw[-,>=stealth] (12) -- (2);
\draw[-,>=stealth] (12) -- (14);

\draw[-,>=stealth] (9) -- (13);
\draw[-,>=stealth] (9) -- (14);

\draw[-,>=stealth] (11) -- (14);

\draw[-,>=stealth] (13) -- (3);

\draw[-,>=stealth] (14) -- (4);
\draw[-,>=stealth] (14) -- (5);

\draw[-,>=stealth] (3) -- (6);
\draw[-,>=stealth] (3) -- (15);

\draw[-,>=stealth] (5) -- (15);

\draw[-,>=stealth] (4) -- (15);

\end{tikzpicture}}
        \caption{Hasse diagram of the covering graph $C_\mathcal{P} = (V, \hat{R})$. Each color represents a part $\pi_i$ of a partition of $V$.}
        \label{fig:theorem_poset}
    \end{subfigure}
    \hfill
    \begin{subfigure}[t]{0.4\textwidth}
        \centering
        \scalebox{.55}{\begin{tikzpicture}

    \node[elementb] (6) at (6,10) {\color{white}6};
    \node[elementb] (3) at (6,8) {\color{white}3};
    \node[elementb] (13) at (6,6)  {\color{white}13};
    \node[elementb] (9) at (6,4) {\color{white}9};
    \node[elementb] (1) at (6,2) {\color{white}1};
    \node[root,fill_black] (16) at (6,0)  {\color{white}16};

    \node[elementw] (5) at (4,8) {5};
    \node[elementw] (10) at (4,4) {10};
    \node[root] (7) at (4,2)  {7};
    \node[elementw] (14) at (2,6) {14};
    \node[elementw] (11) at (2,4) {11};

    \node[elementg2] (15) at (4,10) {15};
    \node[root,fill_gray2] (4) at (2,8)  {4};

    \node[root,fill_gray1] (8) at (2,2)  {8};
    
    \node[elementgm] (2) at (0,4) {2};
    \node[root,fill_graym] (12) at (0,2) {12};

    \draw[-,>=stealth] (16) -- (1);

    \draw[-,>=stealth] (1) -- (9);
    
    \draw[-,>=stealth] (7) -- (10);
    \draw[-,>=stealth] (7) -- (11);

    \draw[-,>=stealth] (12) -- (2);
    
    \draw[-,>=stealth] (9) -- (13);
    
    \draw[-,>=stealth] (11) -- (14);
    
    \draw[-,>=stealth] (13) -- (3);
    
    \draw[-,>=stealth] (14) -- (5);
    
    \draw[-,>=stealth] (3) -- (6);
    
    \draw[-,>=stealth] (4) -- (15);
    
    \end{tikzpicture}}
        \caption{Hasse diagrams of the covering graphs $G_i = (\pi_i, A_i)$ after partitioning $V$. Squared elements represent roots $r_i$.}
        \label{fig:subposet_theorem}
    \end{subfigure}
    
    \begin{subfigure}[t]{0.4\textwidth}
        \centering
        \scalebox{.8}{\begin{tikzpicture}    
    \node[root, fill_black] (1) at (6,0) {\color{white}$\pi_1$};
    \node[elementgm] (2) at (4,2) {$\pi_2$};
    \node[elementg1] (3) at (6,2) {$\pi_3$};
    \node[elementw] (4) at (6,4) {$\pi_4$};
    \node[elementg2] (5) at (6,6) {$\pi_5$};   

    \draw[-,>=stealth] (1) -- (2);

    \draw[-,>=stealth] (1) -- (3);
    
    \draw[-,>=stealth] (3) -- (4);
    \draw[-,>=stealth] (2) -- (4);
    
    \draw[-,>=stealth] (4) -- (5);    
\end{tikzpicture}}
        \caption{Hasse diagram of the covering graph of poset $\mathcal{P}_\pi=(\pi, B)$.}
        \label{fig:induced_poset_theorem}
    \end{subfigure}
    \hfill
    \begin{subfigure}[t]{0.4\textwidth}
        \centering
        \scalebox{.4}{\begin{tikzpicture}
    \node[elementg2] (15) at (4,24) {15};
    \node[elementg2] (4) at (2,22)  {4};

    \node[elementw] (5) at (2,20) {5};
    \node[elementw] (14) at (0,18) {14};
    \node[elementw] (11) at (0,16) {11};
    \node[elementw] (10) at (2,16) {10};
    \node[elementw] (7) at (2,14)  {7};

    \node[elementgm] (2) at (4,14) {2};
    \node[elementgm] (12) at (4,12) {12};    
    
    \node[elementg1] (8) at (4,10)  {8};
    
    \node[elementb] (6) at (6,10) {\color{white}6};
    \node[elementb] (3) at (6,8) {\color{white}3};
    \node[elementb] (13) at (6,6)  {\color{white}13};
    \node[elementb] (9) at (6,4) {\color{white}9};
    \node[elementb] (1) at (6,2) {\color{white}1};
    \node[elementb] (16) at (6,0)  {\color{white}16};

    \draw[-,>=stealth] (16) -- (1);
    \draw[-,>=stealth] (1) -- (9);
    \draw[-,>=stealth] (13) -- (3);
    \draw[-,>=stealth] (3) -- (6);
    
    \draw[-,>=stealth,line width=1.5mm] (3) -- (8);
    \draw[-,>=stealth,line width=1.5mm] (8) -- (12);
    
    \draw[-,>=stealth] (12) -- (2);
    \draw[-,>=stealth,line width=1.5mm] (12) -- (7);
    
    \draw[-,>=stealth] (7) -- (10);
    
    \draw[-,>=stealth] (11) -- (14);
    \draw[-,>=stealth] (7) -- (11);
    \draw[-,>=stealth] (11) -- (14);

    \draw[-,>=stealth] (9) -- (13);

    \draw[-,>=stealth] (14) -- (5);

    \draw[-,>=stealth] (13) -- (3);

    \draw[-,>=stealth,line width=1.5mm] (5) -- (4);
    
    \draw[-,>=stealth] (4) -- (15);
    
    \end{tikzpicture}}
        \caption{Hasse diagram of an arboreal extension $\mathcal{A} = (V, R^\prime)$ of poset of Figure \ref{fig:theorem_poset} (jumps in bold).}
        \label{fig:theorem_extension}
    \end{subfigure}
    \caption{Example of the characterization stated in Theorem \ref{thm:characterization}.}
    \label{fig:characterization}
\end{figure}

\begin{proof}	
($\Rightarrow$) Let $C_\mathcal{P} = (V, \hat{R})$ be the covering graph of $\mathcal{P}$ (Figure \ref{fig:theorem_poset}), $\mathcal{A} = (V,R^\prime)$ be an arboreal extension of $\mathcal{P}$ with $s$ jumps, and $C_\mathcal{A} = (V,\hat{R^\prime})$ its covering graph (Figure \ref{fig:theorem_extension}). Let $J \subset R^\prime, |J| = s$ be the set of jumps (bold lines of Figures \ref{fig:theorem_extension}).
Note that the covering graph of an arboreal order is an arborescence, furthermore $J \subset \hat{R^\prime}$.
Now consider another graph $G = (V, \hat{R^\prime} \setminus J)$. This graph is a set of $s+1$ arborescences.
Call each arborescence $G_i = (\pi_i, A_i)$ and $\pi = \{\pi_1, \ldots, \pi_{s+1}\}$ (Figure \ref{fig:subposet_theorem}).

Note that the result of operation $\hat{R^\prime} \setminus J$ represents arcs associated with covering relations of $\mathcal{P}$, hence each $G_i$ corresponds to the covering graph of an induced subposet of $\mathcal{P}$. 
Therefore, each $\pi_i$ induces an arboreal order $\mathcal{P}_i = (\pi_i, R_i)$, where $R_i$ is the transitive closure of $A_i$.
Then property $\ref{it:induced_subposet}$ follows.

For property~\ref{it:rooted_poset}, we only need to show that $\mathcal{P}_\pi = (\pi,B)$ (Figure \ref{fig:induced_poset_theorem}) is antisymmetric, because transitivity and reflexivity are straightforward by construction.
Note that $(\pi_i,\pi_j) \in B$ if and only if $(r_i, r_j) \in R^\prime$, with $r_i$ and $r_j$ being roots of $\pi_i$ and $\pi_j$ (squared elements of Figure \ref{fig:subposet_theorem}), respectively.
Suppose $\mathcal{P}_\pi$ is not antisymmetric. Then there are $(\pi_i, \pi_j), (\pi_j,\pi_i) \in B$ such that $i \neq j$, as consequence $(r_i, r_j), (r_j,r_i) \in R^\prime$, which is a contradiction since $R^\prime$ is antisymmetric.

For property \ref{it:f_function}, let us construct a function $f:\pi \setminus \{r_{\pi}\} \to V$ as follows.
Let $(y_i, r_j) \in J, y_i \in \pi_i,r_j \in \pi_j, i \neq j$. 
Note that, as a consequence of $\mathcal{A}$ being arboreal, $r_j$ covers only $y_i$, and no other $y_j^\prime \in \pi_j$ covers an element of some partition $\pi_k, k \neq j$. Therefore, for each $\pi_j \in \pi \setminus \{r_{\pi}\}$, we can do $f(\pi_j) = y_i$ represent the image set of $f$ (see Table \ref{tab:function}). 
This also proves that $\mathcal{A}_\pi$ is arboreal. 
Now, we are going to prove that $\mathcal{A}_\pi$ is an extension of $\mathcal{P}_\pi$, that is, if $(\pi_i,\pi_j)\in B$, then $(\pi_i,\pi_j)\in B^\prime$. 
Let $r_i$ and $r_j$ be roots of $\pi_i$ and $\pi_j$, respectively and observe that $(\pi_i, \pi_j) \in B$ implies $(r_i, r_j) \in R^\prime$, then $(\pi_i, \pi_j) \in B^\prime$. 

\begin{table}[ht]
    \centering
    \caption{Values for functions $f(\pi_i)$ and $g(\pi_i)$ defined on properties \ref{it:f_function} and \ref{it:recursion} of Theorem \ref{thm:characterization}. These values correspond to the example presented in Figure \ref{fig:characterization}.}
    \begin{tabular}{r|r|r|r|r|l}
\toprule
$i$ & $j$ & $f(\pi_i)$ & $r_i$ & $(f(\pi_i), r_i)$ & $g(\pi_i)$  \\
\midrule
1 & $\nexists$ & $\nexists$ & 16 & $\nexists$ & $\emptyset$ \\ 
2 & 3 & 8 & 12 & (8,12) & $\{f(\pi_2)\} \cup g(\pi_3) = \{3,8\}$ \\ 
3 & 1 & 3 & 8 & (3,8) & $\{f(\pi_3)\} \cup g(\pi_1) = \{3\}$ \\ 
4 & 2 & 12 & 7 & (12,7) & $\{f(\pi_4)\} \cup g(\pi_2) = \{3,8,12\}$ \\ 
5 & 4 & 5 & 4 & (5,4) & $\{f(\pi_5)\} \cup g(\pi_4) = \{3,5,8,12\}$ \\ 

\bottomrule
\end{tabular}

    \label{tab:function}
\end{table}

For property \ref{it:recursion}, let $x \prec_{\mathcal{P}} y, x \in \pi_u, y \in \pi_v, u \neq v$. Recall that $\mathcal{A}$ is an arboreal extension of $\mathcal{P}$ so there is a path from $r$ to $y$ going through $x$. Moreover, since $x$ and $y$ are into different parts, then there is at least one jump in that path. Observe that by the recursive definition of $g$, $g(\pi_v)$ contains all the sources of existing jumps from $\pi_r$ to $\pi_v$, including the one belong to $\pi_u$, which is exactly $z$, and therefore $(x,z) \in \pi_u$ (see Table \ref{tab:function}).

($\Leftarrow$) Let $\mathcal{P} = (V,R)$ be a partial order and $\mathbf{V} = \{\pi_1,\ldots,\pi_{s+1}\}$ be a partition of $V$, satisfying properties \ref{it:induced_subposet}, \ref{it:rooted_poset}, \ref{it:f_function} and \ref{it:recursion}. 
Define the following sets: $A_1 = \{(x,x) \mid x \in V\}$, $A_2 = \{(r, x) \mid x \in V \setminus \{r\}\}$, $A_3=\{(x,y) \mid (x,y) \in R_i, i = 1, \ldots, s+1\}$, $A_4=\{(f(\pi_j), r_j) \mid \pi_j \in \mathbf{V} \setminus \{\pi_1\}, f(\pi_j)\in \pi_i, r_j \in \pi_j, i \neq j\}$, squared elements represent $r_j$ and filled elements represent $f(\pi_j)$). 
Construct the relation $\mathcal{A} = (V,R^\prime)$, such that $R^\prime$ is the transitive closure of $A_1 \cup A_2 \cup A_3 \cup A_4$. Let us show that $\mathcal{A}$ is an arboreal extension of $\mathcal{P}$ with $s$ jumps.

First, we prove that $\mathcal{A}$ is a partial order. It is easy to see that $\mathcal{A}$ is reflexive (set $A_1$) and transitive ($R^\prime$ is a transitive closure). 
Suppose $\mathcal{A}$ is not antisymmetric, so there are $x,y \in V, x \neq y$ such that $\{(x,y),(x,y)\} \subset R^\prime$.
This is only possible, by property \ref{it:induced_subposet}, if $x$ and $y$ belong to different sets of $\pi$, so $x \in \pi_i, y \in \pi_j, i \neq j$.
Note that, by property \ref{it:rooted_poset}, when constructing $\mathcal{Q}$ we have $\{(\pi_i,\pi_j), (\pi_j,\pi_i)\} \subset \hat{B}$ and therefore $\mathcal{P}_\pi$ is not a partial order, a contradiction.
So $\mathcal{A}$ must be antisymmetric.

We know that the covering graph of $\mathcal{A}$, $C_\mathcal{A} = (V,\hat{R^\prime})$, must be an arborescence to be an arboreal order.
Evidently, $C_{\mathcal{A}}$ has an unique minimal element, which is $r$ (set $A_2$). 
Assume $C_{\mathcal{A}}$ is not an arborescence, so there are distinct elements $u,v,z$ that form a violation ($\{(u,v),(z,v)\} \subset \hat{R^\prime}$).
By property \ref{it:induced_subposet}, $u,v,z$ could not be in the same set $\pi_k \in \pi$.
Assume $v \in \pi_l, u \in \pi_m, z \in \pi_n, m \neq l, n \neq l$, then we would have $(f(\pi_l), v) \in A_4$ and $f(\pi_l) = u = z$, a contradiction. 
Now, consider $\{u,v\} \in \pi_m, z \in \pi_n, m \neq n$, so we would have $(f(\pi_m), v) \in A_4$, but, since $(u,v) \in R_m$, $v$ is not the root of $\pi_m$, which contradicts property \ref{it:f_function}. 
The same happens if we exchange $u$ and $z$.
Therefore, $\mathcal{A}$ is arboreal.

Now, observe that $\text{\textbar}A_4\text{\textbar} = s$ and if $(x,y) \in A_4, x \in \pi_i, y \in \pi_j, i \neq j$ then $x \prec_\mathcal{A} y$ and $(x,y) \notin R$ by property~\ref{it:f_function}. 
Hence each $(x,y) \in A_4$ is a jump of $\mathcal{A}$.

Finally, we need to prove that $\mathcal{A}$ is an extension of $\mathcal{P}$, that is, $R \subseteq R^\prime$.
Let $\hat{R}$ be the transitive reduction of $R$, it is enough to show that $\hat{R} \subseteq R^\prime$. 
Note that if $\{x,y\} \in \pi_i$ then $(x,y) \in A_3$.
Now, consider that $(x,y) \in \hat{R}$ and $x \in \pi_i, y \in \pi_j, i \neq j$.
We need to show that there is a path from $x$ to $y$ in $R^\prime$.
First, recall that, by property \ref{it:induced_subposet}, $(r_j, y) \in R^\prime$, since $(r_j,y) \in A_3$ or $(r_j,y) \in A_1$, if $r_j = y$. 
Furthermore, by property \ref{it:recursion}, there is $z = f(\pi_u) \in g(\pi_j)$, such that $z \in \pi_i$ and $\{(x,z),(z,r_j)\} \in R^\prime$. 
Moreover, there is a sequence of sources of jumps belonging to $g(\pi_j)$ and as a consequence we can obtain a path that contains those elements from $z$ to $r_j$, joining the three relations $(x,z)$, $(z,r_j)$, $(r_j,y)$.

\end{proof}

\section{Solution strategies}
\label{sec:strategies}
In this section, we present an integer programming formulation for the arboreal jump number problem.  We also present a new greedy heuristic based on an algorithm to find minimal arboreal extensions \cite{figueiredo2013}.

\subsection{Integer programming formulation}
\label{sec:flow_model}
Let $C_\mathcal{P} = (V, \hat{R})$ be the covering graph of the partial order $\mathcal{P}$ and let $E = \hat{R} \cup Z$ be a set such that $Z = \{ (i, j), (j, i) \mid i, j \in V, i \| j\}$. Let $c_{ij}, (i,j) \in E$ be $1$ if $(i,j) \in \hat{R}$ and $0$ if $(i,j) \in Z$.
Let 
\begin{equation*}
x_{ij}=\begin{cases}
         1 & \text{if } (i,j) \in E \text{ is in the solution} \\
         0 & \text{otherwise.} \\
     \end{cases}
\end{equation*}
Let $f_{ij}^k$ be the amount of flow that goes through the arc $(i,j)$ having vertex $k$ as destination.

The following model formulates the arboreal jump number problem.

\begin{equation}
\max \sum_{ij}c_{ij}x_{ij},
\label{eq:objective}
\end{equation}
subject to:

\begin{flalign}
\qquad & \sum_{i \in N^-(j)} f_{ij}^j = 1 & \forall j \in V - \{r\}\label{eq:sum_flow_vertex}\\
\qquad & \sum_{i \in N^-(j)} f_{ij}^k = \sum_{i \in N^+(j)} f_{ji}^k & \forall j,k \in V - \{r\}, k \neq j \label{eq:flow_conservation}\\
\qquad & \sum_{i \in N^-(j)} x_{ij} = 1 & \forall j \in V - \{r\} \label{eq:indegree_vertex}\\
\qquad & \sum_{i \in N^-(j)} f_{ij}^k = 1 & \forall (j,k) \in \hat{R}, j\neq r \label{eq:assure_relation}\\
\qquad & f_{ij}^k \leq x_{ij} & \forall (i,j) \in E, \forall k \in V - \{r\} \label{eq:linking_variables}\\
\qquad &  & x_{ij}\in \{0,1\} \label{eq:binary}\\
\qquad & & f_{ij}^k\geq 0 \label{eq:nonnegative}
\end{flalign}

Note that the objective function \ref{eq:objective} maximizes the number of arcs from $\hat{R}$ that is used in the solution. Since the solution is an arborescence, the number of arcs in the solution is constant, then the arboreal jump number is given by the number of arcs belonging to $Z$ that are used in the solution, which is implicitly minimized in our model.

Constraints (\ref{eq:sum_flow_vertex}) assure that each vertex that is not the root gets a unit of flow.
Constraints (\ref{eq:flow_conservation}) are flow conservation constraints, assuring that the flow to a given vertex $k$ is not consumed elsewhere.
Constraints (\ref{eq:indegree_vertex}) restrict the in-degree of each vertex, implying the solution is in fact an arborescence.
For each pair $(j,k) \in \hat{R}$, constraints (\ref{eq:assure_relation}) assure that the flow to vertex $k$ goes through vertex $j$. This fact along with the constraints (\ref{eq:linking_variables}) guarantee that the relation $R$ will be present in the arborescence.
Constraint (\ref{eq:linking_variables}) allow the flow to move only along arcs in the solution. 
Constraints (\ref{eq:binary}) and (\ref{eq:nonnegative}) guarantee the binary and non-negativity properties of the solution, respectively.

\subsection{Heuristic} \label{sec:heuristic}
A \textit{minimal arboreal extension} is an arboreal extension of a poset such that the removal of any of its jumps makes it no longer arboreal.
While the task of finding a minimum arboreal extension is NP-hard, there is a simple polynomial-time algorithm to find a minimal arboreal extension for a given order, detailed in Algorithm \ref{algo:plainMinExt}  \cite{figueiredo2013}. 
The main idea of this algorithm is to iteratively add relations between incomparable elements in the original poset eliminating at least one violation and ensuring that no new violation is created.    

\begin{algorithm}
\caption{Algorithm to find a minimal arboreal extension for a partial order $\mathcal{P}$}
\label{algo:plainMinExt}
\begin{algorithmic}[1]
	\Require partial order $\mathcal{P}$
	\Ensure minimal arboreal extension $\mathcal{A}$ of $\mathcal{P}$
	\State define $\mathtt{numviol}_\mathcal{P}$ as the number of violations of the partial order $\mathcal{P}$
	\State $\mathcal{A} \leftarrow \mathcal{P}$
	\While{$\mathtt{numviol}_\mathcal{A} > 0$}
		\State choose a violation $v$, $x_1$, $x_2$, with $v$ being a violator \label{algo:violation}
		\State let $z$ be a minimal element of $N^-_\mathcal{A}[x_1] \setminus N^-_\mathcal{A}[x_2]$ \label{algo:z}
		\State $\mathcal{A} \leftarrow \mathcal{A} + (x_2,z)$ \label{algo:extP}
	\EndWhile
\end{algorithmic}
\end{algorithm}

In each iteration, a violation $v$, $x_1$ and $x_2$ is selected at line \ref{algo:violation}.
In the next step, at line \ref{algo:z}, the algorithm identifies a minimal element $z$ that precedes element $x_1$ (it may be $x_1$ itself) but does not precede element $x_2$, which means that $x_2$ and $z$ are incomparable.
An extension of $\mathcal{P}$ is generated at line \ref{algo:extP} inserting the jump $(x_2,z)$.
The fact that $z$ is minimal assures that no new violation is added.
The algorithm continues until there is no violations left in the extension. 

Observe that Algorithm \ref{algo:plainMinExt} does not consider any rule regarding the order in which the violations are selected.
However, the use of different rules may affect the final number of the jump-minimal extension $\mathcal{A}$.
For instance, it is preferable to first select a violation in which its removal also removes other violations.

A simple way to select a violation is to sort the elements in some lexicographical order and choose always the first violator in the given order.
This arbitrary order though, could lead to a solution with many jumps, even for simple instances.
Figure \ref{fig:heuristicSolution} illustrates an example of such a case.
The elements of the poset in Figure \ref{fig:ladderPoset} are enumerated in a way that using this simple approach produces the worst possible arboreal extension having four jumps (Figure \ref{fig:badSolution}) , while the optimal solution contains only one jump, as showed in Figure \ref{fig:goodSolution}. 

\begin{figure}[H]
	\centering
	\begin{subfigure}[t]{0.3\textwidth}
		\centering
		\scalebox{.6}{\begin{tikzpicture}
\node[element] (1) at (0,14) {1};
\node[element] (3) at (0,12) {3};
\node[element] (5) at (0,10) {5};
\node[element] (7) at (0,8)  {7};
\node[element] (9) at (0,6) {9};

\node[element] (2) at (2,12) {2};
\node[element] (4) at (2,10)  {4};
\node[element] (6) at (2,8) {6};
\node[element] (8) at (2,6)  {8};
\node[element] (10) at (2,4)  {10};

\draw[-,>=stealth] (3) -- (1);
\draw[-,>=stealth] (5) -- (3);
\draw[-,>=stealth] (7) -- (5);
\draw[-,>=stealth] (9) -- (7);

\draw[-,>=stealth] (4) -- (2);
\draw[-,>=stealth] (6) -- (4);
\draw[-,>=stealth] (8) -- (6);
\draw[-,>=stealth] (10) -- (8);

\draw[-,>=stealth] (10) -- (9);
\draw[-,>=stealth] (8) -- (7);
\draw[-,>=stealth] (6) -- (5);
\draw[-,>=stealth] (4) -- (3);
\draw[-,>=stealth] (2) -- (1);
\end{tikzpicture}}
		\caption{Hasse diagram of a poset with 10 nodes.}
		\label{fig:ladderPoset}
	\end{subfigure}	
	~
	\begin{subfigure}[t]{0.3\textwidth}
		\centering
		\scalebox{.6}{\begin{tikzpicture}
\node[element] (1) at (0,18) {1};
\node[element] (2) at (0,16) {2};
\node[element] (3) at (0,14) {3};
\node[element] (4) at (0,12)  {4};
\node[element] (5) at (0,10) {5};
\node[element] (6) at (0,8) {6};
\node[element] (7) at (0,6)  {7};
\node[element] (8) at (0,4)  {8};
\node[element] (9) at (0,2) {9};
\node[element] (10) at (0,0)  {10};

\draw[-,>=stealth] (1) -- (2);
\draw[-,>=stealth] (3) -- (4);
\draw[-,>=stealth] (5) -- (6);
\draw[-,>=stealth] (7) -- (8);
\draw[-,>=stealth] (9) -- (10);

\draw[-,>=stealth, dashed] (2) -- (3);
\draw[-,>=stealth, dashed] (5) -- (4);
\draw[-,>=stealth, dashed] (6) -- (7);
\draw[-,>=stealth, dashed] (9) -- (8);
\end{tikzpicture}}
		\caption{Hasse diagram of an arboreal extension of poset (\ref{fig:ladderPoset}) with 4 jumps.}
		\label{fig:badSolution}
	\end{subfigure}
	~
	\begin{subfigure}[t]{0.3\textwidth}
		\centering
		\scalebox{.6}{\begin{tikzpicture}
\node[element] (1) at (0,18) {1};
\node[element] (3) at (0,16) {3};
\node[element] (5) at (0,14) {5};
\node[element] (7) at (0,12)  {7};
\node[element] (9) at (0,10) {9};

\node[element] (2) at (0,8){2};
\node[element] (4) at (0,6)  {4};
\node[element] (6) at (0,4) {6};
\node[element] (8) at (0,2)  {8};
\node[element] (10) at (0,0)  {10};

\draw[-,>=stealth] (3) -- (1);
\draw[-,>=stealth] (5) -- (3);
\draw[-,>=stealth] (7) -- (5);
\draw[-,>=stealth] (9) -- (7);

\draw[-,>=stealth] (4) -- (2);
\draw[-,>=stealth] (6) -- (4);
\draw[-,>=stealth] (8) -- (6);
\draw[-,>=stealth] (10) -- (8);

\draw[-,>=stealth, dashed] (9) -- (2);

\end{tikzpicture}}
		\caption{Hasse diagram of an arboreal extension of poset (\ref{fig:ladderPoset}) with 1 jump.}
		\label{fig:goodSolution}
	\end{subfigure}
	\caption{Hasse diagrams of a poset and two arboreal extension obtained with Algorithm \ref{algo:plainMinExt} using lexicographical order (\ref{fig:badSolution}) and Algorithm \ref{algo:greedyMinExt} (\ref{fig:goodSolution}).}
	\label{fig:heuristicSolution}
\end{figure}
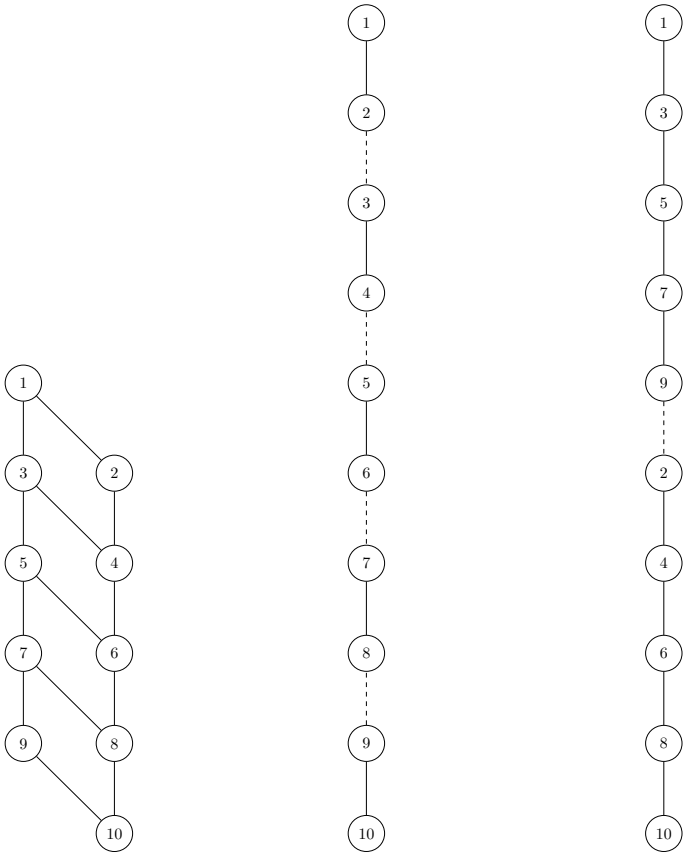

A natural property of an element in a poset is its hierarchical level, or just level.
In case of rooted posets, it is the number of arcs of the longest path in the covering graph between the root and the element.
During the process of removing a given violation, there is a chance to remove precedent violations belonging to lowermost levels. 
Observe that it is the case in the aforementioned example.
In Algorithm \ref{algo:greedyMinExt}, we propose a greedy heuristic considering the level of the elements and the number of violations that precede the violator.

\begin{algorithm}[H]
	\caption{Heuristic for the minimum arboreal extension for a partial order $\mathcal{P}$ using level criterion to select violations}
	\label{algo:greedyMinExt}
	\begin{algorithmic}[1]
		\Require a partial order $\mathcal{P}$
		\Ensure a minimal arboreal extension $\mathcal{A}$ of $\mathcal{P}$
		\State define $\mathtt{numviol}_\mathcal{P}$ as the number of violations of the partial order $\mathcal{P}$
		\State define $\mathtt{lvl}_\mathcal{P}(v)$ as the level of element $v$, such that $v$ is a violator in $\mathcal{P}$
		\State for each $x$ and $v$ such that  $x \prec_\mathcal{P} v$, define $\mathtt{numpred}_\mathcal{P}(x,v)$ as the number of violators $v^\prime$ such that  and $v^\prime \leq_\mathcal{P} x$ 
		\State $\mathcal{A} \leftarrow \mathcal{P}$
		
		\While{$\mathtt{numviol}_\mathcal{A} > 0$}
		\State select violator $v$ such that $\mathtt{lvl}_\mathcal{A}(v)$ is maximum \label{algo:v}
		\State select element $x_1$ such that $\mathtt{numpred}_\mathcal{A}(x_1,v)$ is maximum \label{algo:x1}
		\State select $x_2$ and $z$ such that: (i)~$x_2 \neq x_1, x_2  \prec_\mathcal{A} v$, (ii)~$z$ is a minimal element of $N^-_\mathcal{A}[x_1] \setminus N^-_\mathcal{A}[x_2]$ and (iii)~the number of violations in $\mathcal{A} + (x_2,z)$ is minimum \label{algo:x2}
		\State $\mathcal{A} \leftarrow \mathcal{A} + (x_2,z)$ \label{algo:jump}
		\EndWhile
	\end{algorithmic}
\end{algorithm}

Algorithm \ref{algo:greedyMinExt} iterates until there are no more violations in the extension and it follows three rules to select the violation $v$, $x_1$ and $x_2$ at each iteration.
By the first rule, at line \ref{algo:v}, the choice of $v$ assures that the level of the violator is maximum. 
After selecting $v$, the algorithm considers the set of possible choices for $x_1$.
Then, it counts the number of violators that precedes each element of that set. $x_1$ is selected as the element with the maximum number of violators that precedes it.
Line \ref{algo:x1} represents this choice and states the second rule.
Finally, at line (line \ref{algo:x2}), the algorithm tests the insertion of all possible pairs $(x_2,z)$ to guarantee that it adds the one that removes the largest number of violations (line \ref{algo:jump}).

\section{Computational results}
\label{sec:results}
In this section, we evaluate the performance of the two solutions strategies for the arboreal jump number problem described previously. 

\subsection{Computational environment}
The integer programming formulation was implemented using the \texttt{C++} API of
\texttt{ILOG CPLEX} solver, release $20.1$ \cite{cplex}.
We set up the maximum running time limit of the solver to $7200$ seconds, and the number of threads to four.
It uses default options for all other features offered by \texttt{ILOG CPLEX}.
The greedy heuristic was also implemented in \texttt{C++}.

Both solutions were compiled using \texttt{g++} under the Ubuntu operating system, release 18.04.6 LTS.
We executed the computational experiments in a computer equipped with a Core i7 980 processor with 6 core running at 1333Mhz and having 24Gb of RAM memory.

\subsection{Benchmark instances}
The computational results were conducted over a set of $41$ instances for the sequential ordering problem available in \texttt{TSPLIB} \cite{tsplib}, and a root was inserted in each instance.
The first column of Table \ref{tab:instances} gives the instance, the next two columns describe the instances in terms of its number of vertices $\text{\textbar}V\text{\textbar}$, the number of pairs of incomparable elements $\frac{\text{\textbar}Z\text{\textbar}}{2}$, recall that $\text{\textbar}Z\text{\textbar}$ contains two arcs for each incomparable pair. Finally, the last two columns present the number of violators $\mathtt{numvior}$ and the number of violations $\mathtt{numvion}$ of the instance. 

\begin{table}[htbp]
\centering
\begin{scriptsize}
\begin{tabular}{l|rrrrrr}
\toprule
Instance	&	$\text{\textbar}V\text{\textbar}$	&	$\text{\textbar}\hat{R}\text{\textbar}$	&	$\frac{\text{\textbar}Z\text{\textbar}}{2}$	&	$\mathtt{numvior}$	&	$\mathtt{numvion}$	\\
\midrule
br17.10	&	18	&	18	&	105	&	1	&	1	\\
br17.12	&	18	&	19	&	98	&	2	&	2	\\
esc07	&	9	&	10	&	14	&	1	&	3	\\
esc11	&	13	&	12	&	50	&	0	&	0	\\
esc12	&	14	&	18	&	55	&	2	&	9	\\
esc25	&	27	&	27	&	289	&	1	&	1	\\
esc47	&	49	&	49	&	1049	&	1	&	1	\\
esc63	&	65	&	149	&	1720	&	10	&	405	\\
esc78	&	80	&	79	&	2720	&	0	&	0	\\
ft53.1	&	54	&	55	&	1314	&	2	&	2	\\
ft53.2	&	54	&	58	&	1296	&	5	&	5	\\
ft53.3	&	54	&	70	&	1109	&	12	&	25	\\
ft53.4	&	54	&	75	&	567	&	18	&	27	\\
ft70.1	&	71	&	73	&	2329	&	3	&	3	\\
ft70.2	&	71	&	80	&	2298	&	8	&	12	\\
ft70.3	&	71	&	96	&	2131	&	17	&	38	\\
ft70.4	&	71	&	101	&	1021	&	26	&	36	\\
kro124p.1	&	101	&	103	&	4818	&	3	&	3	\\
kro124p.2	&	101	&	113	&	4783	&	10	&	17	\\
kro124p.3	&	101	&	131	&	4585	&	25	&	38	\\
kro124p.4	&	101	&	150	&	2546	&	41	&	59	\\
p43.1	&	44	&	46	&	850	&	3	&	3	\\
p43.2	&	44	&	46	&	827	&	3	&	3	\\
p43.3	&	44	&	54	&	765	&	8	&	15	\\
p43.4	&	44	&	59	&	365	&	11	&	22	\\
prob.100	&	100	&	99	&	4712	&	0	&	0	\\
prob.42	&	42	&	42	&	761	&	1	&	1	\\
rbg048a	&	50	&	197	&	681	&	14	&	1009	\\
rbg050c	&	52	&	260	&	717	&	46	&	1031	\\
rbg109a	&	111	&	629	&	557	&	93	&	1924	\\
rbg150a	&	152	&	963	&	841	&	140	&	3173	\\
rbg174a	&	176	&	1116	&	1096	&	166	&	3844	\\
rbg253a	&	255	&	1724	&	1697	&	245	&	6140	\\
rbg323a	&	325	&	2416	&	3801	&	303	&	13044	\\
rbg341a	&	343	&	2547	&	3667	&	311	&	11309	\\
rbg358a	&	360	&	3243	&	7367	&	354	&	16557	\\
rbg378a	&	380	&	3073	&	7668	&	369	&	14820	\\
ry48p.1	&	49	&	49	&	1069	&	1	&	1	\\
ry48p.2	&	49	&	52	&	1055	&	4	&	4	\\
ry48p.3	&	49	&	62	&	949	&	9	&	22	\\
ry48p.4	&	49	&	70	&	485	&	15	&	30	\\
\bottomrule
\end{tabular}

\end{scriptsize}
\caption{Description of \texttt{SOP} instances in terms of number of elements $\text{\textbar}V\text{\textbar}$, number of relations in the covering graph $\text{\textbar}\hat{R}\text{\textbar}$, number of incomparable pairs $\frac{\text{\textbar}Z\text{\textbar}}{2}$, number of violator $\mathtt{numvior}$ and number of violations $\mathtt{numvion}$.}
\label{tab:instances}
\end{table}

\subsection{Results}

Table \ref{tab:results} reports the results for each instance obtained by the three strategies presented in this work.
Again, the first column shows the instance names.
The second column shows the cost of the best solution obtained with the integer programming model within the time limit.
The values in bold corresponds to proven optimal costs.
The next two columns show the wall time and the final gap whenever the time limit was not reach.
The next column shows the cost of the solution found by Algorithm \ref{algo:plainMinExt} \cite{figueiredo2013}, as described in section \ref{sec:heuristic}.
The last two columns present the cost of the solution found by Algorithm \ref{algo:greedyMinExt} introduced in this work and its wall time.

\begin{table}[htbp]
\centering
\begin{scriptsize}
\begin{tabular*}{\textwidth}{l|rrr|r|rr}
\toprule
 & \multicolumn{3}{|c}{Integer programming model} & \multicolumn{1}{|c}{Algo \cite{figueiredo2013}} & \multicolumn{2}{|c}{Heuristic}\\
\cmidrule{2-7}
Instance	&	\multicolumn{1}{c}{Best}	&	\multicolumn{1}{c}{Wall time (s)}	&	\multicolumn{1}{c}{Gap (\%)}	&	\multicolumn{1}{|c}{Best}	&	\multicolumn{1}{|c}{Best}	&	\multicolumn{1}{c}{Wall time (s)}	\\
\midrule
\textbf{br17.10}	&	\textbf{1}	&	0.11	&	0.00	&	1	&	1	&	$< 0.01$	\\
\textbf{br17.12}	&	\textbf{2}	&	0.14	&	0.00	&	2	&	2	&	$< 0.01$	\\
\textbf{esc07}	&	\textbf{2}	&	0.01	&	0.00	&	2	&	2	&	$< 0.01$	\\
\textbf{esc11}	&	\textbf{0}	&	0.00	&	0.00	&	0	&	0	&	$< 0.01$	\\
\textbf{esc12}	&	\textbf{5}	&	0.08	&	0.00	&	5	&	5	&	$< 0.01$	\\
\textbf{esc25}	&	\textbf{1}	&	0.35	&	0.00	&	1	&	1	&	$< 0.01$	\\
\textbf{esc47}	&	\textbf{1}	&	7.79	&	0.00	&	1	&	1	&	$< 0.01$	\\
esc63	&	-	&	7200.00	&	-	&	34	&	34	&	0.03	\\
\textbf{esc78}	&	\textbf{0}	&	2.37	&	0.00	&	0	&	0	&	$< 0.01$	\\
\textbf{ft53.1}	&	\textbf{2}	&	8.15	&	0.00	&	2	&	2	&	$< 0.01$	\\
ft53.2	&	6	&	7200.00	&	4.26	&	5	&	5	&	$< 0.01$	\\
ft53.3	&	-	&	7200.00	&	-	&	14	&	14	&	$< 0.01$	\\
ft53.4	&	-	&	7200.00	&	-	&	18	&	14	&	0.01	\\
\textbf{ft70.1}	&	\textbf{3}	&	39.22	&	0.00	&	3	&	3	&	$< 0.01$	\\
ft70.2	&	15	&	7200.00	&	10.91	&	10	&	10	&	$< 0.01$	\\
ft70.3	&	-	&	7200.00	&	-	&	24	&	19	&	0.01	\\
ft70.4	&	-	&	7200.00	&	-	&	23	&	17	&	0.02	\\
\textbf{kro124p.1}	&	\textbf{3}	&	219.65	&	0.00	&	3	&	3	&	$< 0.01$	\\
kro124p.2	&	-	&	7200.00	&	-	&	13	&	12	&	$< 0.01$	\\
kro124p.3	&	-	&	7200.00	&	-	&	28	&	23	&	0.02	\\
kro124p.4	&	-	&	7200.00	&	-	&	32	&	27	&	0.06	\\
\textbf{p43.1}	&	\textbf{3}	&	8.71	&	0.00	&	3	&	3	&	$< 0.01$	\\
\textbf{p43.2}	&	\textbf{3}	&	7.46	&	0.00	&	3	&	3	&	$< 0.01$	\\
p43.3	&	8	&	7200.00	&	5.71	&	11	&	8	&	$< 0.01$	\\
\textbf{p43.4}	&	\textbf{9}	&	2743.47	&	0.00	&	12	&	10	&	$< 0.01$	\\
\textbf{prob.100}	&	\textbf{1}	&	2.49	&	0.00	&	1	&	1	&	$< 0.01$	\\
\textbf{prob.42}	&	\textbf{0}	&	0.32	&	0.00	&	0	&	0	&	0.01	\\
rbg048a	&	-	&	7200.00	&	-	&	32	&	32	&	0.03	\\
rbg050c	&	-	&	7200.00	&	-	&	36	&	33	&	0.05	\\
rbg109a	&	-	&	7200.00	&	-	&	86	&	75	&	0.84	\\
rbg150a	&	-	&	7200.00	&	-	&	125	&	107	&	3.51	\\
rbg174a	&	-	&	7200.00	&	-	&	146	&	117	&	6.19	\\
rbg253a	&	-	&	7200.00	&	-	&	210	&	167	&	25.75	\\
rbg323a	&	-	&	7200.00	&	-	&	279	&	236	&	87.15	\\
rbg341a	&	-	&	7200.00	&	-	&	293	&	243	&	99.86	\\
rbg358a	&	-	&	7200.00	&	-	&	299	&	245	&	134.27	\\
rbg378a	&	-	&	7200.00	&	-	&	335	&	270	&	167.09	\\
\textbf{ry48p.1}	&	\textbf{1}	&	5.04	&	0.00	&	1	&	1	&	$< 0.01$	\\
\textbf{ry48p.2}	&	\textbf{4}	&	193.94	&	0.00	&	4	&	4	&	$< 0.01$	\\
ry48p.3	&	-	&	7200.00	&	-	&	13	&	12	&	$< 0.01$	\\
ry48p.4	&	-	&	7200.00	&	-	&	14	&	13	&	0.01	\\
\bottomrule
\end{tabular*}

\end{scriptsize}
\caption{Computational results for the presented solutions strategies and the literature algorithm for finding a minimal arboreal extension. The values in bold are proven to be optimal.} 
\label{tab:results}
\end{table}

The integer programming model is able to find an optimality certificate for $18$ instances and a feasible solution for other three instances within the time limit.
These instances are the smallest ones according to the number of violators and the size of set $Z$.
For the other $20$ instances the model could not found any solution on the provided time limit.

The heuristic, on the other hand, found feasible solutions instantaneously for most instances, taking up to tree minutes to solve instances with more than $300$ violators.
For all but one instance, the solution cost found by the heuristic was at least as good as the one found by the integer programming model.
Instance \texttt{p43.4} was proven to have an optimal arboreal extension within nine jumps. However, the heuristic found a solution with ten jumps and Algorithm \ref{algo:plainMinExt} obtained a solution containing 12 jumps. The results for this instance exemplifies a case in which the heuristic may not find the optimal solution but still find a better solution than the minimal arboreal extension algorithm.

Finally, the proposed heuristic was able to find a better solution for $19$ out of $41$ instances than the algorithm from the literature, mostly the harder instances, obtained a difference of up $65$ jumps between both algorithms.

\section{Conclusions}
\label{sec:conclusions}
In this study, we proved that the ground set of an arboreal extension with $k$ jumps can be partitioned into $k+1$ sets that follow four structural properties.
Additionally, we also presented an exact and a heuristic strategies to solve the arboreal jump number problem and reported some computational results. 

Our integer programming model based on multi-flow was able to provide optimal solution for $18$ out of $41$ instances within the running time limit provided.
Our proposed greedy heuristic incorporates rules for selecting violations to be removed to a polynomial-time algorithm to find minimal arboreal extensions.
It was able to find feasible solutions for all instances in less than three minutes. 
The solutions are at least as good, and in most cases better than the ones obtained with the algorithm from the literature.

As far as we know, this is the first study with computational results for this problem.
The heuristic shows very good results for instances where the optimal value is known.
The performance of the heuristic on harder instances can not be assessed since optimal solutions are not available for those instances.

An research direction to improve this work points out to prove some lower bound to the problem considering the number of violators belonging to the input poset. Consequently it would be possible to determine the quality of the solutions given by the described heuristic.
Furthermore, we plan to investigate valid inequalities for
the arboreal jump number problem in order to develop a more efficient exact algorithm.


\bibliography{refs}


\end{document}